\documentclass{amsart}[12pt]
\usepackage{mathrsfs, amsmath,amssymb}
\usepackage{fullpage}
\newtheorem{teo}{Theorem}

\newtheorem{cor}{Corollary}
\theoremstyle{remark}
\newtheorem{rem}{Remark}

\title{Generalized Bessel functions of dihedral-type: expression as a series of confluent Horn functions and Laplace-type integral representation}
\date{Septembre 2018}
\begin{document}

\author[Luc Deleaval]{L. Deleaval}
\address{Laboratoire d'Analyse et de Math\'ematiques appliqu\'ees \\ Universit\'e Paris-Est Marne-la-Vall\'ee \\
France}
\email{luc.deleaval@u-pem.fr}

\author[N. Demni]{N. Demni}
\address{IRMAR, Universit\'e de Rennes 1\\ Campus de
Beaulieu\\ 35042 Rennes cedex\\ France}
\email{nizar.demni@univ-rennes1.fr}
\maketitle
\begin{abstract}
In the first part of this paper, we express the generalized Bessel function associated with dihedral systems and a constant multiplicity function as a infinite series of confluent Horn functions. The key ingredient leading to this expression is an extension of an identity involving Gegenbauer polynomials proved in a previous paper by the authors, together with the use of the Poisson kernel for these polynomials. In particular, we derive an integral representation of this generalized Bessel function over the standard simplex. The second part of this paper is concerned with even dihedral systems and boundary values of one of the variables. Still assuming that the multiplicity function is constant, we obtain a Laplace-type integral representation of the corresponding generalized Bessel function, which extends to all even dihedral systems a special instance of the  Laplace-type integral representation proved in \cite{Amr-Dem}.  
\end{abstract}

\section{Introduction}

Generalized Bessel functions associated with dihedral groups have received considerable attention in recent times. This is mainly due to their occurence in several branches of mathematics, such as harmonic analysis or representation theory. Indeed, beyond their most natural setting of being the symmetric counterpart of the so-called Dunkl kernel, they turn out to be surprisingly connected to the Laguerre semi-group constructed in  \cite{BKO} and in the particular case of the square-preserving dihedral group, to certain representations of the indefinite orthogonal group of rank two \cite{Kob-Man}. Besides, special instances of them are Laplace transforms of Duistermaat-Heckman measures which were introduced in \cite{BBO} for finite Coxeter groups by means of generalized Pitman transforms and given there a very interesting probabilistic interpretation.

One of the challenging problems concerning generalized Bessel functions associated with dihedral groups is to find relatively simple formulas in terms of well-known special functions, and to potentially obtain integral representations of Laplace-type for them. Results towards this goal have been recently obtained in a series of papers (\cite{Amr-Dem,CDBL,DDY,Del-Dem,Dem2,Xu}). In particular, the identity recalled below in \eqref{IdGeg} and proved in \cite{Del-Dem} shows that, if one of the two variables of the generalized Bessel function lies on the boundary of the dihedral wedge and if the multiplicity function is constant, then it may be expressed through the confluent Horn function $\Phi_2$. 

In this paper, we shall pursue this line of research and prove, only assuming that the multiplicity function is constant, that the generalized Bessel function is given by a infinite series of confluent Horn functions. Our main tool is an extension of the aforementioned identity to the case of two Gegenbauer polynomials, and its proof appeals to their Poisson kernel. As a by-product, we obtain an integral representation over the standard simplex in an Euclidean space whose dimension is half of the order of the underlying dihedral group. 
Assuming further that the latter is even and that one of the variables lies on the boundary of the dihedral wedge, we shall derive a Laplace-type integral representation of the generalized Bessel function involving the standard simplex of an Euclidean space of smaller dimension (one quarter of the order of the underlying dihedral group). Up to our best knowledge, this kind of integral representations is only derived for the root systems of type $A$ (\cite{Amr}) or $B_2$ (\cite{Amr-Dem}), the latter being a particular dihedral root system. In this respect, it is worth noting that Theorem 1 in \cite{Amr-Dem} motivates the extension of our Laplace-type integral representation to all even dihedral groups without any assumption neither on the multiplicity function nor on the locations of the variables. 

The paper is organized as follows. In the next section, we recall some basic facts on dihedral groups and some definitions of special functions we will need later on. In section three, we express the generalized Bessel function associated with a given dihedral group
 and a constant multiplicity function as a series of confluent Horn functions, from which we deduce an integral representation for it over a simplex. In the last section, we prove a Laplace-type integral representation for the generalized Bessel function associated with even dihedral group  and assuming in addition that one of its variables lies on the boundary of the dihedral wedge.

\section{Background and notations} 
In this section, we recall the expression of the generalized Bessel functions of dihedral-type and introduce some special functions occuring in the sequel. For the interested reader, we refer to \cite{Dun-Xu} for a good account on general root systems and Dunkl operators to which generalized Bessel functions are canonically associated and to \cite{AAR} for the definitions and the properties of one-variable special functions occurring below.  

The dihedral group $\mathcal{D}_2(n), n \geq 3$, consists of orthogonal transformations leaving invariant a regular $n$-sided polygon centered at the origin. 
As a finite reflection group, $\mathcal{D}_2(n)$ corresponds to the dihedral root system
\begin{equation*}
\mathcal I_2(n) := \{\pm i \mathrm{e}^{i\pi l/n},\, 1 \leq l \leq n\}. 
\end{equation*}  
Choosing the positive root system
\begin{equation*}
\{- i \mathrm{e}^{i\pi l/n},\, 1 \leq l \leq n\}, 
\end{equation*}  
then the simple system consists of the vectors $\{\alpha_1 :=i, \alpha_2 := -i\mathrm{e}^{i\pi/n}\}$ so that the positive Weyl chamber $C$ is the dihedral wedge
\begin{equation*}
C := \bigl\{(r, \theta), \, r > 0, \, 0 < \theta < \pi/n\bigr\}.  
\end{equation*}
The dihedral group acts on its root system in a natural way and this action admits two orbits when $n := 2p, p\geq 2,$ (the roots forming the diagonals and those forming the lines joining the midpoints of the polygon) while there is only one orbit when $n$ is odd. Consequently, a multiplicity function $k$ (a function which is constant on each conjugate class) on $\mathcal I_2(n)$ takes two values $(k_0,k_1)$ if $n$ is even and only one value, again denoted by $k$, otherwise. For our later purposes, we only consider 
nonnegative multiplicity values though the formulas below extend to larger domains in the complex plane. 

Let $n=2p, p \geq 2$, then the generalized Bessel functions associated with the even dihedral group $\mathcal{D}_2(n)$ and the multiplicity function $(k_0,k_1)$, which we shall denote by  $D_k^{\mathcal{D}_2(n)}$, admits the following expansion in polar coordinates (\cite{Dem0})
\begin{equation*}
D_k^{\mathcal{D}_2(n)}(x,y) = c_{n,k_0,k_1}\biggl(\frac{2}{r\rho}\biggr)^{\gamma} \sum_{j \geq 0}{\it I}_{2jp +\gamma}(\rho r)p_j^{(l_1, l_0)}\bigl(\cos(2p\phi)\bigr)p_j^{(l_1, l_0)}\bigl(\cos(2p\theta)\bigr),
\end{equation*} 
where $x = \rho e^{i\phi}, y = re^{i\theta}$ belong to $\overline{C}$ and
\begin{itemize}
\item $\displaystyle{\gamma = p(k_0+k_1),}$
\item $\displaystyle{p_j^{(l_1, l_0)}}$ is the $j$-th orthonormal Jacobi polynomial of parameters \[ l_i = k_i - (1/2), i \in \{0,1\},\] 
\item ${\it I}_{\nu}$ is the modified Bessel function 
of the first kind
\begin{equation}\label{Bessel}
{\it I}_{\nu}(v) = \sum_{m \geq 0} \frac{1}{m!\Gamma(\nu+m+1)} \biggl(\frac{v}{2}\biggr)^{2m + \nu},
\end{equation}
\item $c_{n,k_0,k_1}$ is a normalizing constant subject to $\displaystyle{D^{\mathcal{D}_2(n)}_k(0,y) = 1}$ for all $y \in \overline{C}$. 
\end{itemize} 
Note in passing that $D_k^{\mathcal{D}_2(n)}$ is $\mathcal{D}_2(n)$-invariant and this invariance manifests itself by the presence of $\cos(2p\theta)$ in the argument of Jacobi polynomials. In the particular case $k_0= k_1 = k$, we can rewrite $D_k^{\mathcal{D}_2(n)}$ by means of non orthonormal Gegenbauer polynomials $C_j^{(k)}$ as:
\begin{equation}\label{Eq1}
D_k^{\mathcal{D}_2(n)}(x,y) = \frac{c_{n,k}}{B(k+1/2,1/2)}\biggl(\frac{2}{r\rho}\biggr)^{\gamma} \sum_{j \geq 0}{\it I}_{2jp +\gamma}(\rho r)(j+k) \frac{C_j^{(k)}\bigl(\cos(2p\phi)\bigr)C_j^{(k)}\bigl(\cos(2p\theta)\bigr)}{C_j^{(k)}(1)},
\end{equation} 
where 
\begin{equation*}
C_j^{(k)}(1) = \frac{(2k)_j}{j!}, \quad c_{n,k}:=c_{n,k,k},  
\end{equation*}
and 
\begin{equation*}
B(u,v) = \frac{\Gamma(u)\Gamma(v)}{\Gamma(u+v)}
\end{equation*}
is the Beta function. For odd dihedral groups, the corresponding generalized Bessel function is also given by \eqref{Eq1} where we should identify $2p$ with $n$ so that $\gamma = nk$  \footnote{The second formula displayed in \cite{Dem0}, Corollary 1, is erroneous.}. This coincidence allows to unify both cases (dihedral groups with equal multiplicity values and odd dihedral groups) which is the main focus of the paper. 

In the next section, we shall derive an expansion of $D_k^{\mathcal{D}_2(n)}$ as a series of confluent Horn functions (see for instance \cite[chapter V]{Erd1}, \cite{Kar-Sri}):
\begin{equation}\label{Horn}
\Phi_{2}^{(n)}(\beta_1, \ldots, \beta_{n}; \gamma; z_1, \dots, z_{n}) := \sum_{j_1, \ldots, j_{n} \geq 0} \frac{(\beta_1)_{j_1}\ldots (\beta_{n})_{j_{n}}}{(\gamma)_{j_1+\cdots+j_{n}}} \frac{z_1^{j_1}}{j_1!}\cdots  \frac{z_{n}^{j_{n}}}{j_{n}!},
\end{equation}
where $(\cdot)_j$ is the so-called Pochhammer symbol.
The occurrence of this function is motivated by the special boundary value $\phi = 0$ (a similar statement holds for $\theta = \pi/n$) for which the series displayed in the right-hand side of \eqref{Eq1} reduces to 
\begin{align*}
\sum_{j \geq 0}{\it I}_{jn +\gamma}(\rho r)(j+k)C_j^{(k)}\bigl(\cos(n\theta)\bigr),
\end{align*}
which is equal, up to a factor depending only on $n$ and $k$, to (\cite{CDBL,Del-Dem})
\begin{equation}\label{Horn1}
\Phi_{2}^{(n)}\Biggl(\underbrace{k, \dots, k}_{n \mathrm{\ times}}; nk; \rho r \cos(\theta), \rho r \cos\biggl(\theta + \frac{2\pi}{n}\biggr), \dots, \rho r \cos\biggl(\theta + \frac{2\pi(n-1)}{n}\biggr)\Biggr).
\end{equation}
Actually, we shall see that \eqref{Horn1} is the lowest-order term of the series displayed in Corollary \ref{coco}. 

\section{Generalized Bessel function as a series of confluent Horn functions} 
Using the expansion \eqref{Bessel} of the modified Bessel function, the right-hand side  of \eqref{Eq1} may be written up to the factor $\displaystyle{\frac{c_{n,k}}{nB(k+1/2,1/2)}}$ as a double series
\begin{equation*}
\sum_{j,m \geq 0}\frac{n(j+k)}{m!\Gamma(nj+nk+m+1)} \frac{C_j^{(k)}\bigl(\cos(n\phi)\bigr)C_j^{(k)}\bigl(\cos(n\theta)\bigr)}{C_j^{(k)}(1)}\biggl(\frac{\rho r}{2}\biggr)^{2m + nj}
\end{equation*}
which converges absolutely since 
\begin{equation*}
\bigl|C_j^{(k)}(z)\bigr| \leq C_j^{(k)}(1) = \frac{(2k)_j}{j!}, \quad |z| \leq 1. 
\end{equation*}
Hence, we will focus on the following finite sums
\begin{equation*}
S_N(n,k, \phi, \theta) := \sum_{\substack{j,m\geq0\\ N = 2m+nj}} \frac{n(j+k)}{m!\Gamma\bigl(n(j+k) +m+ 1\bigr)} \frac{C_j^{(k)}\bigl(\cos(n\phi)\bigr)C_j^{(k)}\bigl(\cos(n\theta)\bigr)}{C_j^{(k)}(1)}, \qquad N \geq 0.
\end{equation*}
If $\phi = 0$, then Proposition 1 in \cite{Del-Dem} asserts that 
\begin{equation}\label{IdGeg}
S_N(n,k, 0, \theta) = \frac{2^N}{\Gamma(nk+N)}\sum_{\substack{j_1, \ldots, j_n \geq 0 \\ j_1+\cdots +j_n = N}}  
(k)_{j_1}\ldots (k)_{j_{n}} \frac{\bigl({b^{\theta,n}_1}\bigr)^{j_1}}{j_1!}\cdots  \frac{\bigl({b^{\theta,n}_n}\bigr)^{j_1}}{j_{n}!},
\end{equation}
where we have set
\begin{equation*}
b^{\theta,n}_s : = \cos\biggl(\theta + \frac{2\pi s}{n}\biggr), \quad s = 1, \ldots, n. 
\end{equation*}
More generally, we shall prove the following formula.
\begin{teo} \label{teooo}
Let $N \geq 0$. Then, we have the equality 
\begin{multline*}
S_N(n,k, \phi, \theta) = \frac{2^N}{\Gamma(nk+N)}\times \\ \sum_{\substack{m_1, \dots, m_n \geq 0 \\ m_1+\dots +m_n = N}} 
\sum_{j=0}^{\inf(m_1,\dots, m_n)}\frac{(k)_j(k)_j}{(2k)_j j!} \Bigl(-2^{2-n}\sin(n\theta) \sin(n\phi)\Bigr)^j \prod_{s=1}^n \frac{(k+j)_{m_s-j}}{(m_s-j)!}\bigl(b^{\theta-\phi,n}_s\bigr)^{m_s-j}.
\end{multline*}
\end{teo}

\begin{proof}
Multiply $S_N(n,k, \phi, \theta) $ by the factor $\Gamma(nk +N)z^N/2^N$ for small enough $|z|$. Summing the resulting expression over $N \geq 0$ and using the duplication formula
\[
(x)_{2l}=2^{2l}\biggl(\frac{x}{2}\biggr)_{l}\biggl(\frac{1+x}{2}\biggr)_{l},
\]
we get the series 
\begin{align*}
\sum_{m, j \geq 0} \frac{n(j+k)\Gamma\bigl(n(j+k)\bigr)}{m!2^{nj} \Gamma\bigl(n(j+k) +m+ 1\bigr)} \bigl((nj+nk)/2\bigr)_m\bigl((nj+nk+1)/2\bigr)_m \frac{C_j^{(k)}\bigl(\cos(n\phi)\bigr)C_j^{(k)}\bigl(\cos(n\theta)\bigr)}{C_j^{(k)}(1)}z^{2m+nj}.
\end{align*}
Summing the latter over $m$, we further get
\begin{multline*}
\sum_{N \geq 0} S_N(n,k, \phi, \theta) \frac{\Gamma(nk +N)z^N}{2^N} = \\  \sum_{j \geq 0} \frac{z^{nj}}{2^{nj}} \frac{C_j^{(k)}\bigl(\cos(n\phi)\bigr)C_j^{(k)}\bigl(\cos(n\theta)\bigr)}{C_j^{(k)}(1)} 
{}_2F_1\left(\frac{nj+nk}{2}, \frac{nj+nk+1}{2}, nj + nk + 1; z^2\right),
\end{multline*}
where ${}_2F_1$ stands for the Gauss hypergeometric function.
Thanks to the following formula (see for instance \cite{Erd1}, p.101)
\begin{equation*}
{}_2F_1\left(\frac{nj+nk}{2}, \frac{nj+nk+1}{2}, nj + nk + 1; z\right) = \frac{2^{nk+nj}}{\bigl(1+\sqrt{1-z}\bigr)^{nk+nj}}, 
\end{equation*}
valid for $z \in \mathbb{C} \setminus [1,\infty[$, we are led to 
\begin{equation}\label{Poisson}
\frac{2^{nk}}{\bigl(1+\sqrt{1-z^2}\bigr)^{nk}} \sum_{j \geq 0} \frac{z^{nj}}{\bigl(1+\sqrt{1-z^2}\bigr)^{nj}} \frac{C_j^{(k)}\bigl(\cos(n\phi)\bigr)C_j^{(k)}\bigl(\cos(n\theta)\bigr)}{C_j^{(k)}(1)}.
\end{equation}
But the Poisson kernel expression (formula $(19)$ in \cite{Mai}, see also \cite{Dun})
\begin{multline*}
\sum_{j \geq 0}  \frac{C_j^{(k)}\bigl(\cos(n\phi)\bigr)C_j^{(k)}\bigl(\cos(n\theta)\bigr)}{C_j^{(k)}(1)}t^j =  \frac{1}{\Bigl(1-2t\cos\bigl(n(\theta-\phi)\bigr) + t^2\Bigr)^k}  {}_2F_1\left(k, k, 2k; \frac{-4t\sin(n\theta) \sin(n\phi)}{1-2t\cos\bigl(n(\theta-\phi)\bigr) + t^2}\right)
\end{multline*}
which converges absolutely for $|t| < 1$, shows that \eqref{Poisson} may be written after some simplifications as
\begin{multline}\label{Lastfor}
\frac{2^{nk}}{\Bigl(\bigl(1+\sqrt{1-z^2}\bigr)^n + \bigl(1-\sqrt{1-z^2}\bigr)^n -2z^n \cos\bigl(n(\theta-\phi)\bigr)\Bigr)^k}\times \\ {}_2F_1\left(k, k, 2k; \frac{-4z^n \sin(n\theta) \sin(n\phi)}{\bigl(1+\sqrt{1-z^2}\bigr)^n + \bigl(1-\sqrt{1-z^2}\bigr)^n -2z^n \cos\bigl(n(\theta-\phi)\bigr)^k}\right). 
\end{multline}
Expanding the hypergeometric function and using the identity (\cite{AAR})
\begin{align*}
 \bigl(1+\sqrt{1-z^2}\bigr)^n + \bigl(1-\sqrt{1-z^2}\bigr)^n = 2 \sum_{j=0}^{[n/2]} \binom{n}{2j}(1-z^2)^j = 2z^nT_n\biggl(\frac{1}{z}\biggr) ,
 \end{align*}
where $T_n$ is the well-known $n$-th Tchebycheff polynomial of the first kind, then formula \eqref{Lastfor} reads 
\begin{equation*}
2^{nk} \sum_{j \geq 0} \frac{(k)_j(k)_j}{(2k)_j j!} \frac{\bigl(-4z^n \sin(n\theta) \sin(n\phi)\bigr)^j}{\Bigl(2z^n T_n\left(1/z\right) - 2z^n \cos\bigl(n(\theta-\phi)\bigr)\Bigr)^{j+k}}.
\end{equation*}
Moreover, the following factorization holds (see \cite{Del-Dem})
\begin{equation*}
2z^n T_n\left(1/z\right) - 2z^n \cos\bigl(n(\theta-\phi)\bigr) = 2^n \prod_{s=1}^n \Bigl(1 - b^{\theta-\phi,n}_s z\Bigr)
\end{equation*}
whence we get the identity
\begin{multline*}
2^{nk} \sum_{j \geq 0} \frac{(k)_j(k)_j}{(2k)_j j!} \frac{\bigl(-4z^n \sin(n\theta) \sin(n\phi)\bigr)^j}{\Bigl(2z^n T_n\left(1/z\right) - 2z^n \cos\bigl(n(\theta-\phi)\bigr)\Bigr)^{j+k}} = \\
\sum_{j \geq 0} \frac{(k)_j(k)_j}{(2k)_j j!} \frac{\bigl(-z^n \sin(n\theta) \sin(n\phi)\bigr)^j}{2^{(n-2)j}} \prod_{s=1}^n \Bigl(1 - b^{\theta-\phi,n}_s z\Bigr)^{-(k+j)}.
\end{multline*}
Now, the generalized binomial theorem entails
\begin{align*}
z^{nj} \prod_{s=1}^n \Bigl(1 - b^{\theta-\phi,n}_s z\Bigr)^{-(k+j)} = \sum_{m_1, \dots, m_n \geq 0} z^{nj+m_1+\dots+m_n}\prod_{s=1}^n \frac{(k+j)_{m_s}}{m_s!}\bigl(b^{\theta-\phi,n}_s\bigr)^{m_s}
\end{align*}
so that 
\begin{multline*}
\sum_{j \geq 0} \frac{(k)_j(k)_j}{(2k)_j j!} \frac{\bigl(-z^n \sin(n\theta) \sin(n\phi)\bigr)^j}{2^{(n-2)j}} \prod_{s=1}^n\Bigl(1 - b^{\theta-\phi,n}_s z\Bigr)^{-(k+j)}= \\ 
\sum_{j, m_1, \dots, m_n \geq 0} \frac{(k)_j(k)_j}{(2k)_j j!}z^{(m_1+j) + \dots+ (m_n+j)} \frac{\bigl(-\sin(n\theta) \sin(n\phi)\bigr)^j}{2^{(n-2)j}}\prod_{s=1}^n \frac{(k+j)_{m_s}}{m_s!}\bigl(b^{\theta-\phi,n}_s\bigr)^{m_s}.
\end{multline*} 
Finally, we perform in the multiple series above the index changes $m_s+j \rightarrow m_s$, for fixed $j \geq 0$ and each $1\leq s \leq n$,   to obtain
\begin{multline*}
\sum_{j \geq 0} \frac{(k)_j(k)_j}{(2k)_j j!} \frac{\bigl(-z^n \sin(n\theta) \sin(n\phi)\bigr)^j}{2^{(n-2)j}} \prod_{s=1}^n \Bigl(1 - b^{\theta-\phi,n}_s z\Bigr)^{-(k+j)}= \\ 
\sum_{j \geq 0} \sum_{m_1, \dots, m_n \geq j } \frac{(k)_j(k)_j}{(2k)_j j!}z^{m_1 + \dots+ m_n} \frac{\bigl(-\sin(n\theta) \sin(n\phi)\bigr)^j}{2^{(n-2)j}}\prod_{s=1}^n \frac{(k+j)_{m_s-j}}{(m_s-j)!}\bigl(b^{\theta-\phi,n}_s\bigr)^{m_s-j},
\end{multline*} 
and we now change the summation order to end up with
\begin{multline*}
\sum_{N \geq 0} S_N(n,k, \phi, \theta) \frac{\Gamma(nk +N)z^N}{2^N} = 
 \\  \sum_{m_1, \dots, m_n \geq 0} z^{m_1 + \dots+ m_n} 
 \sum_{j=0}^{\inf(m_1,\dots, m_n)}\frac{(k)_j(k)_j}{(2k)_j j!} \Bigl(-2^{2-n}\sin(n\theta) \sin(n\phi)\Bigr)^j\prod_{s=1}^n \frac{(k+j)_{m_s-j}}{(m_s-j)!}\bigl(b^{\theta-\phi,n}_s\bigr)^{m_s-j}.
\end{multline*}
Comparing equal powers of $z$, we are done. 
\end{proof}


As a first corollary, we get the following expansion of the generalized Bessel function.
\begin{cor} \label{coco}
Let $\mathcal{D}_2(n), n \geq 3,$ be a dihedral group with a constant multiplicity function $k$. Then, the associated generalized Bessel function is given by
\begin{multline*}
D_k^{\mathcal{D}_2(n)}(x,y) =  \frac{c_{n,k}}{nB(k+1/2,1/2)\Gamma(nk)} \times \\ \sum_{j \geq 0}\frac{(k)_j(k)_j}{(2k)_j(nk)_{nj} j!} \Biggl(-4\biggl(\frac{\rho r}{2}\biggr)^n\sin(n\theta) \sin(n\phi)\Biggr)^j  \Phi_2^{(n)}\Bigl(k+j, \dots, k+j; nk+nj; \rho r b^{\theta-\phi,n}_1, \dots, \rho r b^{\theta-\phi,n}_n\Bigr).
\end{multline*}
\end{cor}
\begin{proof}
By the very definition of the Pochhammer symbol, the result of the previous theorem can be rewritten as
\begin{multline*}
S_N(n,k, \phi, \theta) =\frac{2^N}{\Gamma(nk)} \times \\ \sum_{\substack{m_1, \dots, m_n \geq 0 \\ m_1+\dots +m_n = N}} \frac{1}{(nk)_N} 
 \sum_{j=0}^{\inf(m_1,\dots, m_n)}\frac{(k)_j(k)_j}{(2k)_j j!}\Bigl(-2^{2-n}\sin(n\theta) \sin(n\phi)\Bigr)^j \prod_{s=1}^n \frac{(k+j)_{m_s-j}}{(m_s-j)!}\bigl(b^{\theta-\phi,n}_s\bigr)^{m_s-j}.
\end{multline*}
Multiplying both sides by $(\rho r/2)^N$ and then summing over $N \geq 0$, it follows that 
\begin{multline*}
\sum_{N \geq 0} S_N(n,k, \phi, \theta) \biggl(\frac{\rho r}{2}\biggr)^N = \frac{1}{\Gamma(nk)} \sum_{m_1, \dots, m_n \geq 0}\frac{(\rho r)^{m_1+\dots+m_n}}{(nk)_{m_1+\dots+m_n}}  \\
\sum_{j=0}^{\inf(m_1,\dots, m_n)}\frac{(k)_j(k)_j}{(2k)_j j!}\Bigl(-2^{2-n}\sin(n\theta) \sin(n\phi)\Bigr)^j \prod_{s=1}^n \frac{(k+j)_{m_s-j}}{(m_s-j)!}  \bigl(b^{\theta-\phi,n}_s\bigr)^{m_s-j},
\end{multline*}
which is equivalent to
\begin{multline*}
\sum_{N \geq 0} S_N(n,k, \phi, \theta) \biggl(\frac{\rho r}{2}\biggr)^N  = \frac{1}{\Gamma(nk)} \sum_{j \geq 0}\frac{(k)_j(k)_j}{(2k)_j j!} \Bigl(-2^{2-n}\sin(n\theta) \sin(n\phi)\Bigr)^j\\   
\sum_{m_1, \dots, m_n \geq j} \frac{(\rho r)^{m_1+\dots+m_n}}{(nk)_{m_1+\dots+m_n}}\prod_{s=1}^n \frac{(k+j)_{m_s-j}}{(m_s-j)!} \bigl(b^{\theta-\phi,n}_s\bigr)^{m_s-j},
\end{multline*}
and then
\begin{multline*}
\sum_{N \geq 0} S_N(n,k, \phi, \theta) \biggl(\frac{\rho r}{2}\biggr)^N  = \frac{1}{\Gamma(nk)} \sum_{j \geq 0}\frac{(k)_j(k)_j}{(2k)_j(nk)_{nj} j!} \Biggl(-4\biggl(\frac{\rho r}{2}\biggr)^n\sin(n\theta) \sin(n\phi)\Biggr)^j \\ 
\sum_{m_1, \dots, m_n \geq 0} \frac{1}{(nk+nj)_{m_1+\dots+m_n}}\prod_{s=1}^n \frac{(k+j)_{m_s}}{(m_s)!} (\rho r)^{m_s}\bigl(b^{\theta-\phi,n}_s\bigr)^{m_s},
\end{multline*}
where we have used in particular the formula
\[
\frac{1}{(nk)_{m_1+\dots+m_n+nj}}=\frac{1}{(nk)_{nj}}\frac{1}{(nk+nj)_{m_1+\dots+m_n}}.
\]
Keeping in mind \eqref{Horn}, it follows that
\begin{multline*}
\sum_{N \geq 0} S_N(n,k, \phi, \theta) \biggl(\frac{\rho r}{2}\biggr)^N 
= \frac{1}{\Gamma(nk)} \sum_{j \geq 0}\frac{(k)_j(k)_j}{(2k)_j(nk)_{nj} j!} \Biggl(-4\biggl(\frac{\rho r}{2}\biggr)^n\sin(n\theta) \sin(n\phi)\Biggr)^j 
\\ \Phi_2^{(n)}\Bigl(k+j, \dots, k+j; nk+nj; \rho r b^{\theta-\phi,n}_1, \dots, \rho r b^{\theta-\phi,n}_n\Bigr),
\end{multline*}
which coincides with $D_k^{\mathcal{D}_2(n)}(x,y)$ up to the normalizing factor $\displaystyle{\frac{c_{n,k}}{nB(k+1/2,1/2)}}$. 
\end{proof}

We now give another corollary, which establishes an integral representation of the generalized Bessel function over the $(n-1)$-dimensional standard simplex, which will be denoted by
\begin{equation*}
\Sigma_n := \biggl\{\Bigl(u_1, \ldots, u_{n-1},\bigl(\underbrace{1-u_1-\cdots-u_{n-1}}_{:=u_0}\bigr)\Bigr) \in \mathbb R^n:\  \, u_0, u_1, \ldots, u_{n-1} \geq 0\biggr\}.
\end{equation*}
\begin{cor}Let $\mathcal{D}_2(n), n \geq 3,$ be a dihedral group with a constant multiplicity function $k$. Then, the generalized Bessel function admits the following integral representation 
\begin{multline*}
D_k^{\mathcal{D}_2(n)}(x,y) =  \frac{c_{n,k}}{nB(k+1/2,1/2)\bigr(\Gamma(k)\bigl)^n}
\int_{\Sigma_n}\Biggl(\exp\biggl(\rho r \Bigl(u_0\cos(\theta-\phi)+\sum_{s=1}^{n-1} u_s b^{\theta-\phi,n}_s\Bigr)\biggr)\times \\ {}_0F_{n-1}\Biggl(2k, k, \dots, k; -4\biggl(\frac{\rho r}{2}\biggr)^nu_0u_1\dots u_{n-1} \sin(n\theta) \sin(n\phi)\Biggr)u_0^{k-1}\prod_{s=1}^{n-1} u_s^{k-1}\Biggr) du_1 \ldots du_{n-1},
\end{multline*}
where  ${}_0F_{n-1}$ is the following hypergeometric function
\begin{equation*}
{}_0F_{n-1}(a_1, \dots, a_{n-1}; z) := \sum_{j \geq 0}\frac{1}{(a_1)_j \dots (a_{n-1})_j} \frac{z^j}{j!},  \quad a_1, \dots, a_{n-1} \in \mathbb{R} \setminus \left(-\mathbb{N}\right), z \in \mathbb{C}.
\end{equation*}
\end{cor}
\begin{proof}
Recall first the Dirichlet integral formula for the simplex: for any $\beta_1, \ldots, \beta_n > 0$, then we have
\begin{equation}\label{Dir}
\int_{\Sigma_n}u_0^{\beta_n-1}\prod_{s=1}^{n-1} u_s^{\beta_s-1} du_1 \ldots du_{n-1} = \frac{\Gamma(\beta_1) \ldots \Gamma(\beta_n)}{\Gamma(\beta_1+\cdots+\beta_n)}.
\end{equation}
Substituting $\beta_s = k+ j+m_s$ in \eqref{Dir} for each $1 \leq s \leq n$, we get
\begin{equation*}
\frac{(k+j)_{m_1}\dots (k+j)_{m_n}}{(nk+nj)_{m_1+\dots+m_n}} = \frac{\Gamma(nk+nj)}{\bigl(\Gamma(k+j)\bigr)^n} \int_{\Sigma_n}u_0^{k+j+m_n-1} \prod_{s=1}^{n-1} u_s^{k+j+m_s-1}du_1 \ldots du_{n-1},
\end{equation*} 
whence 
\begin{multline*}
\Phi_2^{(n)}\Bigl(k+j, \dots, k+j; nk+nj; \rho r b^{\theta-\phi,n}_1, \dots, \rho r b^{\theta-\phi,n}_n\Bigr) = \\ \frac{\Gamma(nk+nj)}{\bigl(\Gamma(k+j)\bigr)^n}\int_{\Sigma_n}\Biggl(\exp\biggl(\rho r \Bigl(u_0b^{\theta-\phi,n}_n+\sum_{s=1}^{n-1} u_s b^{\theta-\phi,n}_s\Bigr)\biggr)u_0^{k+j-1}\prod_{s=1}^n u_s^{k+j-1} \Biggr)du_1 \ldots du_{n-1}
\end{multline*}
which can be rewritten as
\begin{multline*}
\Phi_2^{(n)}\Bigl(k+j, \dots, k+j; nk+nj; \rho r b^{\theta-\phi,n}_1, \dots, \rho r b^{\theta-\phi,n}_n\Bigr)= 
 \\ \frac{(nk)_{nj}\Gamma(nk)}{\bigl((k)_j\bigr)^n\bigr(\Gamma(k)\bigl)^n} \int_{\Sigma_n}\Biggl(\exp\biggl(\rho r \Bigl(u_0b^{\theta-\phi,n}_n+\sum_{s=1}^{n-1} u_s b^{\theta-\phi,n}_s\Bigr)\biggr)u_0^{k+j-1}\prod_{s=1}^n u_s^{k+j-1} \Biggr)du_1 \ldots du_{n-1}.
\end{multline*}
If we use this last equality in the formula of Corollary \ref{coco} and since
\begin{multline*}
\sum_{j \geq 0}\frac{1}{(2k)_j\underbrace{(k)_j\dots (k)_j}_{n-2 \mathrm{\ times}} j!} \Biggl(-4\biggl(\frac{\rho r}{2}\biggr)^nu_0u_1\dots u_{n-1} \sin(n\theta) \sin(n\phi)\Biggr)^j = \\ {}_0F_{n-1}\Biggl(2k, k, \dots, k; -4\biggl(\frac{\rho r}{2}\biggr)^nu_0u_1\dots u_{n-1} \sin(n\theta) \sin(n\phi)\Biggr)
\end{multline*}
by the very definition of ${}_0F_{n-1}$, the corollary is proved.
\end{proof}

\section{Laplace-type integral representation of the generalized Bessel function for boundary angles} 
In this section, we derive a Laplace-type integral representation of the generalized Bessel function associated with even dihedral groups $\mathcal{D}_2(2p)$ and a constant multiplicity function when one of the variables, say $x$, lies on the boundary of the 
dihedral wedge. Remarkably, we shall see that the range of integration is $\Sigma_p$ and not $\Sigma_{2p}$ as expected.

 Due to the symmetry relation $C_j^{(k)}(-z) = (-1)^j C_j^{(k)}(z)$, it suffices to consider the value $\phi = 0$ for which the generalized Bessel function reduces to
\begin{align*}
D_k^{\mathcal{D}_2(n)}(x,y) & = \frac{c_{n,k}}{B(k+1/2,1/2)}\sum_{j,m \geq 0}\frac{j+k}{m!\Gamma\bigl(2p(j+k) +m+ 1\bigr)} C_j^{(k)}\bigl(\cos(2p\theta)\bigr) \biggl(\frac{\rho r}{2}\biggr)^{2m+2jp}
\\& = \frac{c_{n,k}}{nB(k+1/2,1/2)} \sum_{N \geq 0} \Biggl(\sum_{\substack{j,m \geq 0 \\N = jp+m}} \frac{n(j+k)}{m!\Gamma\bigl(2p(j+k) +m+ 1\bigr)} C_j^{(k)}\bigl(\cos(2p\theta)\bigr)\Biggr) \biggl(\frac{\rho r}{2}\biggr)^{2N}.
\end{align*}
The key ingredient is the following result which partially extends a previous one due to the second author valid for $p=2$ and without any assumption either on the multiplicity function or on the location of the  variables (\cite{Dem2}, section 3). It involves the following normalized modified Bessel function of the first kind
\begin{equation*}
i_{\nu}(v) :=\Gamma(\nu+1)\biggl(\frac{2}{v}\biggr)^\nu I_\nu(v)= \sum_{m\geq 0} \frac{1}{m!(\nu+1)_m}\biggl(\frac{v}{2}\biggr)^{2m}, \quad \nu, v \in \mathbb{R}.
\end{equation*}
\begin{teo}
Let $\mathcal{D}_2(2p), p \geq 1,$ be an even dihedral group with a constant multiplicity function. Assume $x= \rho \geq 0$. Then, for any $y = re^{i\theta} \in \overline{C}$,  we have 
\begin{multline*}
 D_k^{\mathcal{D}_2(2p)}(x,y) = \frac{c_{2p,k}\Gamma(pk)}{2pB(k+1/2,1/2)\Gamma(2pk)\bigl(\Gamma(k)\bigr)^p}\times \\ \int_{\Sigma_p}\Biggl(i_{pk-1/2}\Biggl(\frac{\rho}{\sqrt{2}}\sqrt{\bigl(y_1^2+y_2^2\bigr)+\bigl(y_1^2-y_2^2\bigr)u_0+ \bigl(y_1^2-y_2^2\bigr)\sum_{s=1}^{p-1} u_sA_s + 2y_1y_2 \sum_{s=1}^{p-1} u_sB_s}\Biggr)\\ u_0^{k-1}\prod_{s=1}^{p-1}  u_s^{k-1}\Biggr)du_1\dots du_{p-1},
\end{multline*}
where we have set
\begin{equation*}
A_s :=b_s^{0,p}=\cos\biggl(\frac{2s\pi}{p}\biggr), \quad  B_s :=-\sin\biggl(\frac{2s\pi}{p}\biggr), \quad s=1,\ldots,p-1.
\end{equation*}
\end{teo}
\begin{proof}
Recall formula \eqref{IdGeg} (this is Proposition 1 in \cite{Del-Dem}): for two integers  $n,M$ such that $n\geq 1, M \geq 0$  and for $\xi \in [0,\pi]$, we have
 \begin{equation}\label{IdGegb}
\sum_{\substack{m,j \geq 0 \\ M = 2m+nj}}\frac{n(j+k)}{m!\Gamma\bigl(n(j+k)+m+1\bigr)}C_j^{(k)}(\cos \xi) = \frac{2^{M}}{\Gamma(nk+M)}\sum_{\substack{j_1, \ldots, j_n \geq 0 \\ j_1+\cdots +j_n = M}}  
(k)_{j_1}\ldots (k)_{j_{n}} \frac{\bigl(b^{\xi/n,n}_1\bigr)^{j_1}}{j_1!}\cdots  \frac{\bigl(b^{\xi/n,n}_{n}\bigr)^{j_{n}}}{j_{n}!}.
\end{equation}
In particular, if $M = 2N, n = 2p, \xi = 2p\theta$, then \eqref{IdGegb} specializes to
\begin{align*}
\sum_{\substack{m,j \geq 0 \\ N = m+pj}}\frac{2p(j+k)}{m!\Gamma\bigl(2p(j+k)+m+1\bigr)}C_j^{(k)}\bigl(\cos(2p\theta)\bigr) = \frac{2^{2N}}{\Gamma(2N+2pk)} \sum_{\substack{j_1, \ldots, j_{2p} \geq 0 \\ j_1+\cdots +j_{2p} = 2N}}  
\prod_{s=1}^{2p} \frac{(k)_{j_s}}{j_s!} \bigl(b^{\theta,2p}_{s}\bigr)^{j_s}.
\end{align*}
Now, since $\cos(a+\pi) = -\cos(a)$, then the finite sum in the right-hand side of the last equality may be written as
\begin{equation}\label{Sum1}
\sum_{\substack{j_1, \ldots, j_{2p} \geq 0 \\ j_1+\cdots+ j_{2p} = 2N}} (-1)^{j_1+\dots+j_p} \prod_{s=1}^{2p} \frac{(k)_{j_s}}{j_s!}  \prod_{s=1}^p  \bigl(b^{\theta,2p}_{s}\bigr)^{j_s + j_{s+p}}, 
\end{equation}
which may be further simplified as follows. Consider a $2p$-tuple of integers in the sum \eqref{Sum1} arranged in $p$ pairs
\begin{equation*}
(j_1,j_{p+1}), \ldots, (j_p, j_{2p}), \quad \sum_{s=1}^{2p} j_s = 2N, 
\end{equation*}
and assume that one pair consists of an even and an odd integers. Then, the $2p$-tuple obtained from the previous one by permuting the even and the odd integers in the given pair (while keeping the other pairs unchanged) cancels the latter since the expression
\begin{equation*}
\prod_{s=1}^{2p} \frac{(k)_{j_s}}{j_s!}  \prod_{s=1}^p  \bigl(b^{\theta,2p}_{s}\bigr)^{j_s + j_{s+p}}
\end{equation*}
is the same for both tuples. As a matter of fact, the only tuples that remains in \eqref{Sum1} after cancellations are those for which $j_s+j_{s+p}$ is even, therefore 
\begin{multline}\label{Sum2}
\sum_{\substack{j_1, \ldots, j_{2p} \geq 0 \\ j_1+\cdots +j_{2p} = 2N}} (-1)^{j_1+\dots+j_p} \prod_{s=1}^{2p} \frac{(k)_{j_s}}{j_s!}  \prod_{s=1}^p  \bigl(b^{\theta,2p}_{s}\bigr)^{j_s + j_{s+p}} \\ = 
\sum_{\substack{j_1, \ldots, j_{2p} \geq 0 \\ j_1+\cdots +j_{2p} = 2N,\   j_s+j_{s+p} \textrm{\ is even}}}  (-1)^{j_1+\dots+j_p}  \prod_{s=1}^{2p} \frac{(k)_{j_s}}{j_s!}  \prod_{s=1}^p  \bigl(b^{\theta,2p}_{s}\bigr)^{j_s + j_{s+p}}.
\end{multline}
Setting $2m_s = j_s+j_{s+p}$ and using the identity 
\begin{equation*}
\sum_{j_s = 0}^{2m_s} (-1)^{j_s} \frac{(k)_{j_s}(k)_{2m_s-j_s}}{j_s!(2m_s - j_s)!} = \frac{(k)_{m_s}}{m_s!},
\end{equation*}
which is readily checked by equating generating functions of both sides, then the right-hand side of \eqref{Sum2} may be written as 
\begin{align*}
\sum_{\substack{m_1, \ldots, m_{p} \geq 0 \\ m_1+\cdots +m_{p} = N}} \prod_{s=1}^p  \bigl(b^{\theta,2p}_{s}\bigr)^{2m_s}  \sum_{j_s = 0}^{2m_s} (-1)^{j_s} \frac{(k)_{j_s}(k)_{2m_s-j_s}}{j_s!(2m_s - j_s)!} = 
\sum_{\substack{m_1, \ldots, m_{p} \geq 0 \\ m_1+\cdots +m_{p} = N}} \prod_{s=1}^p  \bigl(b^{\theta,2p}_{s}\bigr)^{2m_s} \frac{(k)_{m_s}}{m_s!}.
\end{align*}
Consequently, 
\begin{equation*}
\sum_{N = jp+m} \frac{2p(j+k)}{m!\Gamma\bigl(2p(j+k) +m+ 1\bigr)} C_j^{(k)}\bigl(\cos(2p\theta)\bigr) =  \frac{2^{2N}}{\Gamma(2N+2pk)} 
\sum_{\substack{m_1, \ldots, m_{p} \geq 0 \\ m_1+\cdots +m_{p} = N}} \prod_{s=1}^p \bigl(b^{\theta,2p}_{s}\bigr)^{2m_s} \frac{(k)_{m_s}}{m_s!},
\end{equation*}
whence 
\begin{align*}
D_k^{\mathcal{D}_2(2p)}(x,y) = \frac{c_{2p,k}}{2pB(k+1/2,1/2)} \sum_{N \geq 0} \frac{1}{\Gamma(2N+2pk)} 
\sum_{\substack{m_1, \ldots, m_{p} \geq 0 \\ m_1+\cdots +m_{p} = N}}(\rho r)^{2N} \prod_{s=1}^p \bigl(b^{\theta,2p}_{s}\bigr)^{2m_s} \frac{(k)_{m_s}}{m_s!}.
\end{align*}
Thanks to the duplication formula, we derive:
\begin{multline*}
D_k^{\mathcal{D}_2(2p)}(x,y) = \frac{c_{2p,k}}{2pB(k+1/2,1/2)\Gamma(2pk)}\times \\   \sum_{m_1, \dots, m_p \geq 0} \frac{1}{(pk)_{m_1+\dots+m_p} (pk+1/2)_{m_1+\dots + m_p}} 
 \prod_{s=1}^p \bigl(b^{\theta,2p}_{s}\bigr)^{2m_s} \frac{(k)_{m_s}}{m_s!}\biggl(\frac{\rho r}{2}\biggr)^{2m_s},
 \end{multline*}
 or equivalently
 \begin{multline*}
D_k^{\mathcal{D}_2(2p)}(x,y) = \frac{c_{2p,k}\Gamma(pk)}{2pB(k+1/2,1/2)\Gamma(2pk)\bigl(\Gamma(k)\bigr)^p} \times \\  \sum_{m_1, \dots, m_p \geq 0} \frac{\Gamma(k+m_1) \ldots \Gamma(k+m_p)}{\Gamma\bigl((k+m_1)+\cdots+(k+m_p)\bigr)} \frac{1}{(pk+1/2)_{m_1+\dots + m_p}} 
 \prod_{s=1}^p \frac{\bigl(b^{\theta,2p}_{s}\bigr)^{2m_s}}{m_s!}\biggl(\frac{\rho r}{2}\biggr)^{2m_s}.
 \end{multline*}
Moreover, formula \eqref{Dir} entails
\begin{equation*}
\frac{\Gamma(k+m_1) \ldots \Gamma(k+m_p)}{\Gamma\bigl((k+m_1)+\cdots+(k+m_p)\bigr)}=\int_{\Sigma_p}u_0^{k+m_p-1}\prod_{s=1}^{p-1} u_s^{k+m_s-1} du_1 \ldots du_{p-1},\end{equation*}
and as such,
\begin{multline*}
D_k^{\mathcal{D}_2(2p)}(x,y) =  \frac{c_{2p,k}\Gamma(pk)}{2pB(k+1/2,1/2)\Gamma(2pk)\bigl(\Gamma(k)\bigr)^p}\times\\ \int_{\Sigma_p}\Biggl( \sum_{m_1, \dots, m_p \geq 0} \frac{1}{(pk+1/2)_{m_1+\dots + m_p}} 
\Biggl(\prod_{s=1}^p \frac{\bigl(b^{\theta,2p}_{s}\bigr)^{2m_s}}{m_s!}\biggl(\frac{\rho r}{2}\biggr)^{2m_s} \Biggr)u_0^{k+m_p-1}\prod_{s=1}^{p-1}u_s^{k+m_s-1} \Biggr) du_1\ldots du_{p-1}.
 \end{multline*}
By applying the multinomial theorem
 \begin{equation*}
(a_1+\cdots +a_p)^N = \sum_{m_1+\cdots + m_p = N} \frac{N!}{m_1!\ldots m_p!} \prod_{s=1}^p a_s^{m_s},
 \end{equation*}
 we get, leaving behind the notation $b^{\theta,2p}_{s}$,
 \begin{multline*}
D_k^{\mathcal{D}_2(2p)}(x,y) =
\frac{c_{2p,k}\Gamma(pk)}{2pB(k+1/2,1/2)\Gamma(2pk)\bigl(\Gamma(k)\bigr)^p} \times \\ \int_{\Sigma_p}\Biggl( \sum_{N \geq 0} \frac{1}{N! (pk+1/2)_{N}} \biggl(\frac{\rho r}{2}\biggr)^{2N} \Biggl(u_0\cos^{2}(\theta)+\sum_{s=1}^{p-1}  u_s \cos^{2}\biggl(\theta + \frac{s\pi}{p}\biggr)\Biggr)^N u_0^{k-1}\prod_{s=1}^{p-1}u_s^{k-1}\Biggr) du_1\dots du_{p-1},
\end{multline*}
that is to say, thanks to the very definition of $i_\nu$,
\begin{multline*}
D_k^{\mathcal{D}_2(2p)}(x,y) =
\frac{c_{2p,k}\Gamma(pk)}{2pB(k+1/2,1/2)\Gamma(2pk)\bigl(\Gamma(k)\bigr)^p} \times \\ \int_{\Sigma_p}\Biggl(i_{pk-1/2}\Biggl(\rho r\sqrt{u_0\cos^{2}(\theta)+\sum_{s=1}^{p-1}  u_s \cos^{2}\biggl(\theta + \frac{s\pi}{p}}\biggr)\Biggr)u_0^{k-1}\prod_{s=1}^{p-1}  u_s^{k-1}\Biggr) du_1\dots du_{p-1}.
\end{multline*}
Finally, the trigonometric formulas 
\begin{equation*}
\cos^2(\theta) = \frac{1+\cos(2\theta)}{2}, \quad \cos\biggl(2\theta + \frac{2s\pi}{p}\biggr) = A_s \cos(2\theta) + B_s \sin(2\theta), 
\end{equation*} 
together with the fact that $u_1+\dots + u_{p-1} = 1-u_0$ yield
\begin{multline*}
D_k^{\mathcal{D}_2(2p)}(x,y) = \frac{c_{2p,k}\Gamma(pk)}{2pB(k+1/2,1/2)\Gamma(2pk)\bigl(\Gamma(k)\bigr)^p}\times  \\ \int_{\Sigma_p}\Biggl( i_{pk-1/2}\Biggl(\frac{\rho r}{\sqrt{2}}\sqrt{1+u_0\cos(2\theta)+\cos(2\theta) \sum_{s=1}^{p-1 }u_sA_s + \sin(2\theta) \sum_{s=1}^{p-1} u_sB_s}\Biggr)u_0^{k-1}\prod_{s=1}^{p-1}  u_s^{k-1}\Biggr) du_1\dots du_{p-1},
\end{multline*}
or in cartesian coordinates
\begin{multline*}
 D_k^{\mathcal{D}_2(2p)}(x,y) = \frac{c_{2p,k}\Gamma(pk)}{2pB(k+1/2,1/2)\Gamma(2pk)\bigl(\Gamma(k)\bigr)^p}\times \\ \int_{\Sigma_p}\Biggl(i_{pk-1/2}\Biggl(\frac{\rho}{\sqrt{2}}\sqrt{\bigl(y_1^2+y_2^2\bigr)+\bigl(y_1^2-y_2^2\bigr)u_0+ \bigl(y_1^2-y_2^2\bigr)\sum_{s=1}^{p-1} u_sA_s + 2y_1y_2 \sum_{s=1}^{p-1} u_sB_s}\Biggr)\\ u_0^{k-1}\prod_{s=1}^{p-1}  u_s^{k-1}\Biggr)du_1\dots du_{p-1}.
\end{multline*}
\end{proof}

We are now ready to prove  a Laplace-type integral representation for the generalized Bessel function. Before stating it, we introduce, for $u=(u_1,\ldots,u_{p-1},u_0) \in \Sigma_p$,
\begin{equation*}
a := a(u) = \sqrt{1+u_0+ \sum_{s=1}^{p-1} u_sA_s}, \qquad b := b(u) = \frac{1}{a} \sum_{s=1}^{p-1} u_sB_s, 
\end{equation*}
\begin{equation*}
c := c(u) = \frac{1}{a} \sqrt{1- \Biggl(u_0+\sum_{s=1}^{p-1} u_sA_s\Biggr)^2 - \Biggl(\sum_{s=1}^{p-1} u_sB_s\Biggr)^2}.
\end{equation*}
Note that 
\begin{equation*}
\Biggl(u_0+\sum_{s=1}^{p-1} u_sA_s\Biggr)^2 + \Biggl(\sum_{s=1}^{p-1} u_sB_s\Biggr)^2 = \left|\sum_{s=0}^{p-1} u_se^{2i\pi s/p} \right|^2 
\end{equation*}
so that the square-root is well-defined. Note also that $a$ may vanish (for instance when $p$ is even) on a codimension-one subspace of $\Sigma_p$ and that the expressions involved in the integral representation below may be rewritten in a way that they are well-defined for all $u \in \Sigma_p$ (see the first remark after the proof).
\begin{cor}\label{CorLap}
Let $\mathcal{D}_2(2p), p \geq 1,$ be an even dihedral group with a constant multiplicity function which further satisfies $pk > 1/2$. Assume $x= \rho \geq 0$. Then, for any $y = re^{i\theta} \in \overline{C}$, we have
\begin{equation*}
D_k^{\mathcal{D}_2(2p)}(x, y) = \int_{\mathbb R^2} \mathrm{e}^{\langle y, z\rangle} H_p(\rho,  z) dz,
\end{equation*}
where we have set
\begin{multline*}
H_p(\rho,z) := \frac{c_{2p,k}\Gamma(pk)}{2pB(k+1/2,1/2)\Gamma(2pk)\bigl(\Gamma(k)\bigr)^p}\times  \\ \int_{E_{z,\rho, p}} (a(u)c(u)) \biggl(\frac{2}{\rho^2a^2(u)c^2(u)}\biggr)^{pk - (1/2)} \biggl(\frac{\rho^2a^2(u)c^2(u)}{2} - (c(u)z_1)^2 - (a(u)z_2-b(u)z_1)^2\biggr)^{pk-(3/2)} \\ u_0^{k-1} \prod_{s=1}^{p-1}  u_s^{k-1} du_1\ldots du_{p-1},
\end{multline*}
and, for any $z \in \mathbb{R}^2$, 
\begin{equation*}
E_{z,\rho, p} := \biggl\{u \in \Sigma_p:\  \frac{\rho^2a^2(u)c^2(u)}{2} > (c(u)z_1)^2 + (a(u)z_2-b(u)z_1)^2\biggr\}. 
\end{equation*}
 \end{cor}
 \begin{rem}
We conjecture that $H_p(\rho,\cdot)$ is supported in the convex hull of 
\begin{equation*}
\mathcal{D}_2(2p) \rho = \{\rho \mathrm{e}^{is\pi/p}, \, 1 \leq s \leq 2p\}.
\end{equation*}
This conjecture was proved in \cite{Amr-Dem} for $p=2$. More generally, we can prove only the first half of this conjecture. More precisely, recall from Lemma 3.3. in \cite{Kos} that the convex hull of $\mathcal{D}_2(2p) \rho$ is the set 
\begin{equation*}
z_C - \rho \in \mathbb{R}_+\alpha_1 + \mathbb{R}_+\alpha_2,
\end{equation*}
where we recall the simple root vectors $\alpha_1 = i, \alpha_2 = -i\mathrm{e}^{i\pi/n}$, and $z_C$ is the unique representative of $z$ in the closed Weyl chamber $\overline{C}$. Equivalently, this set is characterized by the following inequalities
\begin{eqnarray*} 
z_{C,1} & \leq & \rho, \\ 
 z_{C,1}\cos\biggl(\frac{\pi}{2p}\biggr) + z_{C,2}\sin\biggl(\frac{\pi}{2p}\biggr) & \leq & \rho \cos\biggl(\frac{\pi}{2p}\biggr).  
 \end{eqnarray*}
The first inequality can be proved as follows: if $z$ is such that $E_{z,\rho, p} \neq \emptyset$, then it satisfies   
\begin{equation*}
c^2(u) z_1^2 < \frac{\rho^2a^2(u)c^2(u)}{2} \quad \Leftrightarrow \quad z_1^2 < \frac{\rho^2}{2} \Biggl(1+u_0+ \sum_{s=1}^{p-1} u_sA_s\Biggr).
\end{equation*}
By writing 
\begin{equation*}
(c(u)z_1)^2 + (a(u)z_2-b(u)z_1)^2 = (b^2(u)+c^2(u))\biggl(z_1 - \frac{a(u)b(u)}{b^2(u)+c^2(u)} z_2\biggr)^2  + \frac{a^2(u)c^2(u)}{b^2(u)+c^2(u)} z_2^2 \geq \frac{a^2(u)c^2(u)}{b^2(u)+c^2(u)} z_2^2,
\end{equation*}
we see that 
\begin{equation*}
z_2^2 \leq \frac{\rho^2}{2} \biggl(1-u_0- \sum_{s=1}^{p-1} u_sA_s\biggr).
\end{equation*}
Consequently, $|z|^2 \leq \rho^2$ and, since reflection groups consist of isometries,  $|z_C| = |z|$, so that $z_{C,1} \leq \rho$. 
\end{rem}

We now prove Corollary \ref{CorLap}.
 
\begin{proof}
The derivation of the integral representation is similar to that of Theorem 1 in \cite{Amr-Dem}. More precisely, recall first from \cite{Amr-Dem} the following integral representation: if $pk > 1/2$, then for any $w \in \mathbb{R}^2$,
\begin{equation*}
i_{pk-1/2}\bigl(|w|\bigr) = \int_{|z| < 1} \mathrm{e}^{\langle w, z\rangle}\bigl(1-|z|^2\bigr)^{pk - 3/2} dz,
\end{equation*}
where $|w| = \sqrt{w_1^2+w_2^2}$ is the Euclidean norm. Next, straightforward computations show that
\begin{equation*}
\sqrt{\bigl(y_1^2+y_2^2\bigr)+\bigl(y_1^2-y_2^2\bigr)u_0+ \bigl(y_1^2-y_2^2\bigr)\sum_{s=1}^{p-1} u_sA_s + 2y_1y_2 \sum_{s=1}^{p-1} u_sB_s} = \sqrt{(a(u)y_1+b(u)y_2)^2 + (c(u)y_2)^2}.
\end{equation*}
Consequently, 
\begin{multline*}
i_{pk-1/2}\Biggl(\frac{\rho}{\sqrt{2}}\sqrt{\bigl(y_1^2+y_2^2\bigr)+\bigl(y_1^2-y_2^2\bigr)u_0+ \bigl(y_1^2-y_2^2\bigr)\sum_{s=1}^{p-1} u_sA_s + 2y_1y_2 \sum_{s=1}^{p-1} u_sB_s}\Biggr) = \\ \int_{|z| < 1} \mathrm{e}^{\rho\bigl(a(u)y_1z_1 +y_2(b(u)z_1 + c(u)z_2)\bigr)/\sqrt{2}}\bigl(1-|z|^2\bigr)^{pk - 3/2} dz,
\end{multline*}
which can be rewritten as
\begin{multline*}
i_{pk-1/2}\Biggl(\frac{\rho}{\sqrt{2}}\sqrt{\bigl(y_1^2+y_2^2\bigr)+\bigl(y_1^2-y_2^2\bigr)u_0+ \bigl(y_1^2-y_2^2\bigr)\sum_{s=1}^{p-1} u_sA_s + 2y_1y_2 \sum_{s=1}^{p-1} u_sB_s}\Biggr)  = \\ (a(u)c(u)) \biggl(\frac{2}{\rho^2a^2(u)c^2(u)}\biggr)^{pk - (1/2)} 
\int_{V_{u,\rho, p}} \mathrm{e}^{\langle y, z \rangle} \biggl(\frac{\rho^2a^2(u)c^2(u)}{2} - (c(u)z_1)^2 - (a(u)z_2-b(u)z_1)^2\biggr)^{pk-(3/2)} dz, 
\end{multline*}
where, for each $u \in \Sigma_p$, the set $V_{u,\rho, p}$ is defined by
\begin{equation*}
V_{u,\rho, p} := \biggl\{z \in \mathbb{R}^2, \frac{\rho^2a(u)^2c(u)^2}{2} > (c(u)z_1)^2 + (a(u)z_2-b(u)z_1)^2\biggr\}.
\end{equation*}
Using Fubini Theorem, the sought integral representation follows. 
\end{proof}

\begin{rem}
Using the so-called shift principle, we can derive from Corollary \ref{CorLap} the Laplace-type integral representation of the dihedral Dunkl kernel along the same lines presented in \cite{Amr-Dem} for the dihedral group of order eight. Besides, it would be interesting to exploit the integral representations proved in this paper in order to derive asymptotic results for the generalized Bessel function and its behaviour near the edges of the dihedral Weyl chamber.  
\end{rem}

\end{document}